\documentclass[12pt]{article}
\usepackage{amsmath,amsfonts,amsthm,graphicx}
\usepackage{a4}

\newtheorem{theorem}{Theorem}[section]

\newtheorem{definition}{Definition}[section]
\newtheorem{corollary}{Corollary}[section]

\newcommand{\dblarrow}[3][1.5em]{\,
  \makebox[0pt][l]{\hspace{3pt}\makebox[#1][c]{\raisebox{8pt}{$#2$}}}%
  \makebox[0pt][l]{\raisebox{2pt}{$\xrightarrow{\hspace{#1}}$}}%
  \makebox[0pt][l]{\hspace{3pt}\makebox[#1][c]{\raisebox{-8pt}{$#3$}}}%
  \raisebox{-2pt}{$\xleftarrow{\hspace{#1}}$}\,}

\newcommand{\alpham}{\alpha_\mathrm{m}}
\newcommand{\alphap}{\alpha_\mathrm{p}}
\newcommand{\Asd}{A_{\mathrm{sd}}}
\renewcommand{\b}[1]{\boldsymbol{#1}}
\newcommand{\cF}{\mathcal{F}}
\newcommand{\cG}{\mathcal{G}}

\newcommand{\ddt}{\frac{\mathrm{d}}{\mathrm{d}t}}
\newcommand{\dr}{\dx[r]}
\newcommand{\ds}{\dx[s]}
\newcommand{\dx}[1][x]{\,\mathrm{d}#1}
\newcommand{\gc}{g_\mathrm{c}}
\newcommand{\mum}{\mu_\mathrm{m}}
\newcommand{\mup}{\mu_\mathrm{p}}
\newcommand{\Nchs}{{n_\mathrm{x}}}
\newcommand{\Nf}{n_\mathrm{f}}
\newcommand{\Ps}{P_{\mathrm{s}}}
\newcommand{\Psd}{P_{\mathrm{sd}}}
\newcommand{\R}{\mathbb{R}}
\newcommand{\ttg}[1]{\ensuremath{(\ref{#1})}}
\newcommand{\tautr}{\tau_{\mathrm{tr}}}
\newcommand{\wcF}{\widetilde{\mathcal{F}}}
\newcommand{\wcG}{\widetilde{\mathcal{G}}}
\newcommand{\whm}{{\widehat{m}}}
\newcommand{\whp}{{\hat{p}}}
\newcommand{\wtm}{{\widetilde{m}}}
\newcommand{\wtp}{{\tilde{p}}}

\title{Reduction of chemical systems by delayed quasi-steady state assumptions}
\author{Tom\'a\v s Vejchodsk\'y$^{\mathrm{a,b,}}$\thanks{Corresponding author. {\it Email address}: \texttt{vejchod@math.cas.cz}}, Radek Erban$^\mathrm{a}$, Philip K. Maini$^\mathrm{a}$}
\date{}

\begin{document}
\maketitle

\begin{center}
\vspace{-24pt}
{\small 
$^\mathrm{a}$Wolfson Centre for Mathematical Biology, Mathematical Institute, University of Oxford, Andrew Wiles Building, Radcliffe Observatory Quarter, Woodstock Road, Oxford, OX2 6GG, United Kingdom
\\
$^\mathrm{b}$Permanent address: Institute of Mathematics, Academy of Sciences, \v{Z}itn\'a 25, Praha~1, CZ-115\,67, Czech Republic
}
\end{center}

\begin{abstract}
Mathematical analysis of mass action models of large complex chemical systems
is typically only possible if the models are reduced.
The most common reduction technique is based on
quasi-steady state assumptions.
To increase the accuracy of this technique we propose
\emph{delayed} quasi-steady state assumptions (D-QSSA) which
yield systems of delay differential equations.
We define the approximation based on D-QSSA, prove the corresponding
error estimate, and show how it approximates the invariant manifold.
Then we define a class of well mixed chemical systems and formulate
assumptions enabling the application of D-QSSA.
We also apply the D-QSSA to a model of Hes1 expression 
and to a cell-cycle model to illustrate
the improved accuracy of the D-QSSA with respect to the standard
quasi-steady state assumptions.
\end{abstract}

{\small
\noindent{\bf Keywords.}
Chemical dynamics, mass action, system reduction, ordinary differential equations, delay differential equations, error estimate.
}

\section{Introduction}


The dynamic behaviour of complex chemical and biochemical systems can be
analysed by the mathematical tools of bifurcation analysis.
However, in certain situations the size of the problem
can make these tools impractical or even impossible to use.
In these cases, we may try to reduce the system to make it amenable to analysis
while conserving its dynamic behaviour.

There are already many ideas and methods for model reduction of chemical systems
described in the literature. For example,
a computational singular perturbation reduction method for chemical kinetics
with slow and fast variables was developed in \cite{Lam:1993}
and its recent analysis presented in \cite{ZagKapKap:2004}.
The method of invariant manifold is presented in \cite{Fraser:1988,GorKar:2003,Nguyen:1989,RouFra:1991}.
A global approach to model reduction based on the concept of minimal entropy production and its numerical implementation can be found in \cite{Lebiedz:2004}.
A model reduction technique for multiscale biochemical networks is described
in \cite{RadGorZinLil:2008}.
Model reductions based on quasi-steady state assumptions and variable lumping
are analysed from the point of view of control theory in \cite{LeiHanTuz:2002}.
A method for approximation of the slow manifold of a complex system in cases
when a direct approximation is not possible is presented in \cite{GeaKev:2005}.
A collection of methods for analytical derivation and numerical computation
of the slow invariant manifolds can be found in \cite{GorKarZin:2004}.
Finally, the paper \cite{OkiMav:1998} reviews three general strategies for model reduction of chemical systems: lumping, sensitivity analysis, and time-scale analysis.

Nevertheless, the idea of using time delays for model reduction
of chemical systems has, to our knowledge, not been explored.
Of course, various models of chemical kinetics with time delays exist.
For example, the law of mass action is extended in \cite{Roussel:1996} to allow
for delayed effects.
Delay differential equations are used in \cite{Monk:2003} to model transcriptional delay.
A stochastic algorithm with delays is presented in \cite{Bratsun:2005,RouZhu:2006}. 
However, a systematic approach for model reduction based on delays does not appear to exist.


In the context of mass action models of chemical systems, the standard tool
for model reduction is the quasi-steady state assumption (QSSA).
Mathematically, mass action models are systems of coupled ordinary differential
equations (ODEs) for a number of variables. Based on practical knowledge
of the chemical system, these variables can, in some cases, be split into fast
and slow variables. The application of the QSSA then
replaces the ODEs for the fast variables by algebraic equations,
allowing them to be expressed
algebraically in terms of the slow variables,
and the original system of ODEs can thus be reduced to a system for
only the slow variables.

The QSSA is used extensively.
The first application of the QSSA to chemical kinetic systems dates back to 1913
\cite{Bodenstain:1913,UndCha:1913}.
The QSSA has been analysed many times, see
the review \cite{SegSle:1989} and references therein. 
One of the best known models in this context is the Michaelis--Menten kinetics \cite{MicMen:1913}. 
Probably the first application of the QSSA to this model appeared in \cite{BriHal:1925}.
A suitable change of variables can enhance the quality of the QSSA approximation,
see for example the total QSSA approach \cite{BorBoeSeg:1996}.
The error of the QSSA approximation is analysed in \cite{TurTomPil:1993},
where ODEs for the error are derived.


Studying the error of the QSSA approximation, we note that in the
original system the fast variables
always need a certain amount of time to reach their quasi-steady states.
Therefore, if the quasi-steady state changes (due to change in the slow variables), the corresponding fast variable will reach the new value of the
quasi-steady state with a certain time delay.
On the other hand, if the original system is reduced by the QSSA then the fast variables
stay in their quasi-steady states and the time delay is neglected.
This discrepancy between the original and reduced systems can be naturally
decreased by introducing time delays to the QSSA.
This new approach is called the \emph{delayed quasi-steady state assumption} (D-QSSA).

The system reduced by the D-QSSA has the form of delay differential equations. These are often considered as infinite-dimensional systems, because they are initialized by a function corresponding to the history of the variables in the system. 
Consequently, the space of initial histories and the state space are infinite dimensional, in general. 
On the face of it, the approximation of a finite dimensional system
(of ODEs) by an infinite dimensional system (of delay differential
equations) does not look like a system reduction.
However, in the D-QSSA we initialize the history by a simple constant prolongation of the initial condition of the original ODEs. Therefore, the initial space has a finite dimension and the D-QSSA reduces the state space in the same way as the QSSA.

The idea of D-QSSA has been recently applied
to a particular biochemical system
modelling circadian rhythms \cite{Vej:MB} to illustrate the accuracy of the D-QSSA
approximation. This has been the first attempt to use
the D-QSSA with a specific application and no analysis.
However, the D-QSSA can be defined for a general class of problems and
its error can be rigorously analysed. Therefore, we present below a general
definition of the D-QSSA and its first error estimate.

In Section~\ref{se:D-QSSA} we define the D-QSSA for a general linear ODE and
we formulate and prove the corresponding error estimate.
This is an upper bound for instantaneous error of the D-QSSA approximation.
It implies that in a special case the error decreases exponentially towards
zero as the system evolves in time.
Section~\ref{se:system} assumes a system of two ODEs and shows how its invariant
manifold is approximated by the D-QSSA. 
Section~\ref{se:massaction} considers a general chemical system
and its corresponding mass action ODE model. We explicitly show how to apply the D-QSSA
to such systems and we rigorously state the necessary assumptions.
In Sections~\ref{se:numex} and \ref{se:cellcyc}
we apply both the QSSA and D-QSSA to a model of
expression of Hes1 protein and to a cell-cycle model.
We compare the accuracy of both approximations 
and numerically illustrate various aspects of the D-QSSA.
Finally, Section~\ref{se:concl} discusses the results and draws conclusions.

\section{Delayed quasi-steady state assumption and error estimates}
\label{se:D-QSSA}


The D-QSSA can be applied to an ordinary differential equation of the form
\begin{equation}
\label{eq:ode1}
  \ddt x(t) = f(t) - g(t)x(t), \quad\textrm{for } t \in (0,T),
\end{equation}
where $T>0$ and $g(t)$ is positive in $(0,T)$.
The D-QSSA approximation is defined as follows.
\begin{definition}
\label{de:D-QSSA}
  The delayed quasi-steady state approximation $\tilde x(t)$
  to the solution $x(t)$ of equation \ttg{eq:ode1} is given by
\begin{equation}
  \label{eq:D-QSSA}
  \tilde x(t) = \frac{f(t-\tau(t))}{g(t-\tau(t))},
  \quad\textrm{where } 
  \tau(t) = \frac{1}{g(t)}.
\end{equation}
\end{definition}
Let us emphasize that $\tilde x(t)$ is designed to approximate
the long-time behaviour of $x(t)$.
If $t$ is close to zero then approximation \ttg{eq:D-QSSA} is still defined,
but the quantity $t-\tau(t)$ can be negative.
Therefore, technically, the functions $f$ and $g$ have to be defined for negative values
of $t$ as well. Consequently, for small values of $t$, the approximation $\tilde x(t)$
depends on arbitrary extensions of $f$ and $g$ to negative values. Thus, we cannot
expect good approximation qualities of $\tilde x(t)$ for $t$ close to zero.
This is in agreement with the properties of the standard QSSA
and with the error estimate presented below.

The D-QSSA approximation can be derived in the case of constant $g(t)$ from the following expression for the solution of \ttg{eq:ode1}
with initial condition $x(0)=x^0$:
$$
  x(t) = x^0 \exp[-t\gc] + \int_0^t f(s) \exp[(s-t)\gc] \ds,
$$
where $\gc>0$ stands for the constant value of $g(t)$. The integral in this expression can be
approximated by a one-node quadrature rule
\begin{equation}
  \label{eq:qr}
  \int_0^t f(s) \exp[(s-t)\gc] \ds \approx w f(t-\tau),
\end{equation}
where the quadrature node $t-\tau$ and the quadrature weight $w$ are to be
determined such that this quadrature rule is exact for all \emph{linear functions} $f$.
This can be seen as a quadrature with weighting function $\omega(s) = \exp[(s-t)\gc]$.
Notice that $\omega(t) = 1$, and as $s$ goes from $t$ towards $0$, values $\omega(s)$
decrease exponentially fast towards zero. Thus, the value of integral \ttg{eq:qr} is
influenced mainly by values of $f$ close to $t$. The values of $f$ far from $t$ have
little influence on the value of the integral. Therefore, the simple one-node
quadrature rule provides reasonable accuracy in many cases.
Further, let us note that for $t < \tau$ the approximation $w f(t-\tau)$ depends on values of $f$ outside $(0,t)$. This corresponds to the poor approximation qualities of $\tilde x(t)$ for small values of $t$.

As we mentioned above, quantities $\tau$ and $w$ can be explicitly determined
such that quadrature rule \ttg{eq:qr} is exact for all \emph{linear
functions} $f$. The expressions for $\tau$ and $w$ are complicated, but if we
neglect all terms that decline to zero exponentially with $t$, we obtain
$\tau = 1/\gc$ and $w=1/\gc$. Consequently, the approximation $x(t)\approx w f(t-\tau)$
coincides with \ttg{eq:D-QSSA}.
For the explicit formulae for $\tau$ and $w$ and for more details about
this heuristic derivation, we refer to \cite{Vej:MB}.

Let us note that using a one-node quadrature rule with the node at $t$,
$$
  \int_0^t f(s) \exp[(s-t)\gc] \ds \approx w_0 f(t),
$$
we can determine the weight $w_0$ such that the rule is
exact for all \emph{constant functions} $f$. Neglecting the exponentially
decaying terms, we obtain $w_0 = 1/\gc$ and, thus, the standard QSSA approximation
$x(t) \approx f(t)/\gc$.

To analyse the accuracy of the D-QSSA approximation \ttg{eq:D-QSSA},
we present the following error estimate.
It is an estimate of the difference between
the solution $x(t)$ of equation \ttg{eq:ode1} and its D-QSSA approximation
$\tilde x(t)$ given by \ttg{eq:D-QSSA}. 

\begin{theorem}
\label{th:errest}
Let $T>0$ and $0 < \varepsilon \leq M$ be fixed constants,
$f \in C^2([-1/\varepsilon,T])$, and $g\in C([-1/\varepsilon,T])$.
Let $0 < \varepsilon \leq g(t) \leq M$ for all $t\in[-1/\varepsilon,T]$.
Further, let $x(t)\in C^1([0,T])$ be the solution of the ODE \ttg{eq:ode1}
with initial condition $x(0) = x^0$
and let $\tilde x(t)$ be given by \ttg{eq:D-QSSA}.
Then
\begin{equation}
\label{eq:errest}
  | x(t) - \tilde x(t) | \leq 
    2 \left( \frac{1}{\varepsilon} - \frac{1}{M} \right) \max\limits_{[-1/\varepsilon,t]} |f|
    +
    \frac{1}{\varepsilon^3} \max\limits_{[-1/\varepsilon,t]} |f''|  + Q(t)
\end{equation}
for all $t\in [0,T]$, where the prime denotes the derivative and
$$
  Q(t) = \left[ |x^0|
    + \frac{1}{\varepsilon} \max\limits_{[-1/\varepsilon,t]} |f|
    + \frac{t}{\varepsilon} \max\limits_{[-1/\varepsilon,t]} |f'|
    + \frac{1}{2\varepsilon} \left( \frac{1}{\varepsilon^2} + t^2 \right) \max\limits_{[-1/\varepsilon,t]} |f''|
    \right] \exp(-\varepsilon t).
$$
\end{theorem}
\begin{proof}
  Let $t \in [0,T]$ be fixed. Without loss of generality we set
$\varepsilon = \min_{[-1/\varepsilon,t]} g$ and $M = \max_{[-1/\varepsilon,t]} g$.

First, we split the error $x(t) - \tilde x(t)$ as follows:
\begin{equation}
\label{eq:xtx}
  |x(t) - \tilde x(t)| \leq |x(t) - \hat x(t)| + |\hat x(t) - \tilde x(t)|
\end{equation}
with
$$
  \hat x(t) = \frac{f(t - \tau(t))}{g(t)}.
$$
Note that $\tau(t) \leq 1/\varepsilon$ and, hence,
the second term in this splitting can be easily bounded as
\begin{equation}
\label{eq:hxtx}
 |\hat x(t) - \tilde x(t)| \leq
   \left( \frac{1}{\varepsilon} - \frac{1}{M}  \right) \max\limits_{[-1/\varepsilon,t]} |f|.
\end{equation}

Since the solution of \ttg{eq:ode1} with initial condition $x(0)=x^0$
can be expressed as
$$
  x(t) = x^0 \exp\left[ - G(t) \right]
       + \int_0^t f(s) \exp\left[ G(s)-G(t) \right] \ds,
$$
where $G(t) = \int_0^t g(r) \dr$,
we can estimate
\begin{multline}
\label{eq:xhx}
  |x(t) - \hat x(t)| \leq
    \left| x^0 \exp\left[ -G(t) \right] \right|
  + \left| \int_0^t f(s) \left( \exp[ G(s) - G(t) ] - \exp[ g(t)(s-t) ] \right) \ds \right|
  \\
  + \left| \int_0^t f(s) \exp[ g(t)(s-t) ] \ds - \hat x(t) \right|.
\end{multline}
Clearly, we can bound the first term on the right-hand side of
\ttg{eq:xhx} as
\begin{equation}
\label{eq:x0est}
  \left| x^0 \exp\left[ -G(t) \right] \right| \leq
    |x^0| \exp[ -\varepsilon t ].
\end{equation}
To bound the second term, we consider
\begin{multline*}
  Z(s) = \exp[ G(s) - G(t) ] - \exp[ g(t)(s-t) ]
    = \exp\left[ -\int_s^t g(r) \dr \right] - \exp[ g(t)(s-t) ]
\\
    \leq \exp[ -\varepsilon(t-s) ] - \exp[ -g(t)(t-s) ].
\end{multline*}
Hence,
\begin{multline*}
  \int_0^t Z(s) \ds \leq
  \int_0^t \exp[ -\varepsilon(t-s) ] - \exp[ -g(t)(t-s) ] \ds
\\ =
  \frac{1}{\varepsilon} - \frac{1}{g(t)}
    - \frac{\exp[-\varepsilon t]}{\varepsilon}
    + \frac{\exp[-g(t) t]}{g(t)}
  \leq \frac{1}{\varepsilon} - \frac{1}{g(t)}
  \leq \frac{1}{\varepsilon} - \frac{1}{M}.
\end{multline*}
Similarly, since
$$
  Z(s) \geq \exp[ -M(t-s) ] - \exp[ -g(t)(t-s) ],
$$
we bound
$$
  \int_0^t Z(s) \ds \geq \frac{1}{M} - \frac{1}{\varepsilon}.
$$
Thus, the second term on the right-hand side of \ttg{eq:xhx}
can be estimated as
\begin{multline}
\label{eq:T2}
 \left| \int_0^t
   f(s) \left( \exp[ G(s) - G(t) ] - \exp[ g(t)(s-t) ] \right)
    \ds \right|
 \leq \left( \frac{1}{\varepsilon} - \frac{1}{M} \right) \max\limits_{[0,t]} |f|.
\end{multline}

To estimate the final term on the right-hand side of \ttg{eq:xhx},
we employ the Taylor expansion for $s \in [0,t]$:
$$
  f(s) = f(t-\tau) + f'(t-\tau)(s-t+\tau) + R(s) (s-t+\tau)^2/2,
$$
where $R \in C([0,t])$ and $\tau = 1/g(t)$ agrees with definition
\ttg{eq:D-QSSA}. Hence, we can calculate
\begin{multline*}
I = \int_0^t f(s) \exp[ g(t)(s-t) ] \ds - \hat x(t)
  = Q_1
  + \int_0^t R(s) \frac{(s-t+\tau)^2}{2} \exp[ g(t)(s-t) ] \ds,
\end{multline*}
where $Q_1 = \left[ t f'(t-\tau) - f(t-\tau) \right] \tau \exp[ -g(t) t ]$.
Since $R$ can be bounded as $|R(s)| \leq \max_{[-1/\varepsilon,t]} |f''|$, we can
estimate
\begin{equation}
\label{eq:T3}
  |I| \leq |Q_1| + \frac{\tau^3}{2} \max\limits_{[-1/\varepsilon,t]} |f''| + |Q_2|,
\end{equation}
where $Q_2 = - \frac{1}{2} \tau (\tau^2 + t^2) \exp[ -g(t) t ]
  \max\limits_{[-1/\varepsilon,t]} |f''| $.

Combining \ttg{eq:xtx}, \ttg{eq:hxtx},
\ttg{eq:xhx}, \ttg{eq:x0est}, \ttg{eq:T2}, \ttg{eq:T3},
and using the fact that
$\tau(t) \leq 1/\varepsilon$, we obtain
$$
  |x(t) - \tilde x(t)| \leq
   2 \left( \frac{1}{\varepsilon} - \frac{1}{M} \right)
      \max\limits_{[-1/\varepsilon,t]} |f|
    +
    \frac{1}{2\varepsilon^3} \max\limits_{[-1/\varepsilon,t]} |f''| + Q_3,
$$
where
$Q_3 = |x^0| \exp[ -\varepsilon t ] + |Q_1| + |Q_2|$.
Clearly, $Q_3 \leq Q(t)$.
\end{proof}

Notice that if $f$, and its first two derivatives, are bounded
then the reminder $Q(t)$ tends to zero as $t$ tends to infinity.
Hence, for long times the error estimate \ttg{eq:errest} is dominated by
the first two terms on its right-hand side.
The first term is proportional to the value $1/\varepsilon - 1/M$
which grows with the variation in $g(t)$. This indicates that the D-QSSA
is accurate in the case of $g(t)$ which does not vary by much.
Since the influence
of the historical values of the coefficients $f(t)$ and $g(t)$ on the solution
of the system faints away as time progresses, we may expect good approximation
properties of the D-QSSA even for those $g(t)$ that vary slowly on the time scales
of the delay $\tau(t)$. In this case, the accuracy of the D-QSSA should be preserved 
even if the coefficient $g(t)$ varies considerably on the global scale.
Further notice that the first term vanishes if $g(t)$ is constant
and the second term vanishes if $f(t)$ is a linear function.
In this case, the approximation \ttg{eq:D-QSSA} is asymptotically exact.
We formulate this statement rigorously:
\begin{corollary}
Let the assumptions of Theorem~\ref{th:errest} be satisfied.
Further,
let $g(t) = \gc$ be constant and $f(t)$ linear in $[-1/\varepsilon,T]$. Then
\begin{equation}
\label{eq:errestgf}
  | x(t) - \tilde x(t) | \leq
  \left[ |x^0|
    + \frac{1}{\gc} \max\limits_{[0,t]} |f|
    + \frac{|f'|}{\gc} t
    \right] \exp(-\gc t)
\end{equation}
for all $t\in [0,T]$.
\end{corollary}
\begin{proof}
This is an immediate consequence of Theorem~\ref{th:errest},
because we can consider $\varepsilon = M = \gc$ and we have $f'' = 0$.
\end{proof}
Notice that error estimate \ttg{eq:errestgf} implies that
the error tends to zero as $t$ tends to infinity. Moreover, this decrease
is exponentially fast.

\section{Analysis of a system of two equations}
\label{se:system}

The analysis of the D-QSSA in the case of a single equation \ttg{eq:ode1} indicates the accuracy of this approach. However, the main interest is in systems. Therefore, in this section we analyse the accuracy of the D-QSSA for a system of two ODEs, where the first equation has the form \ttg{eq:ode1} with constant $g(t)$. In particular, we consider the system
\begin{align}
\label{eq:sys2}
\begin{split}
  \ddt x(t) &= f(y(t)) - \frac{1}{\tau} x(t), \\
  \ddt y(t) &= h(x(t),y(t)),
\end{split}  
\end{align}
where $\tau$ is a constant which coincides with the delay defined in \ttg{eq:D-QSSA}.
In this system, $x(t)$ varies on the timescale $\tau t$ while $y(t)$ varies
on the timescale $t$. Therefore, for small $\tau$, it would be conventional to apply the standard QSSA to obtain the invariant manifold of the system.
Our aim is to apply the D-QSSA to the first equation in this system and analyse to which order in $\tau$ the invariant manifold is approximated.

It turns out that in the case of constant $\tau$ the D-QSSA is equivalent 
up to higher-order terms to a first-order correction of the standard QSSA. 
Therefore, we first
introduce the standard QSSA approximation of system \ttg{eq:sys2} as
\begin{align}
 \label{eq:defx0}
\begin{split}
  x_0(t) &= \tau f( y_0(t)), \\ 
  \ddt y_0(t) &= h( x_0(t), y_0(t) ).
\end{split}
\end{align}
The approximation of \ttg{eq:sys2} based on the first-order correction is then defined as
\begin{align}
\label{eq:defhatx}
\begin{split}  
  \hat x(t) &= x_0(t) - \tau^2 f'(y_0(t)) h(x_0(t),y_0(t)),
\\ 
  \ddt \hat y(t) &= h(\hat x(t), \hat y(t)),
\end{split}  
\end{align}
where $f'$ stands for the derivative of $f$.
First, we show that this first-order correction approximates the invariant manifold
of system \ttg{eq:sys2} up to terms quadratic in $\tau$.
\begin{theorem}
The approximation $\hat x, \hat y$ based on the first-order correction \ttg{eq:defhatx} satisfies system \ttg{eq:sys2} up to terms quadratic in $\tau$.
\end{theorem}
\begin{proof}
Let us consider expansions 
\begin{align*}
  x(t) &= x_0(t) + \tau x_1(t) + \mathcal{O}(\tau^2), \\
  y(t) &= y_0(t) + \tau y_1(t) + \mathcal{O}(\tau^2)
\end{align*}
of the exact solution $x,y$ of system \ttg{eq:sys2}.
Substituting these into \ttg{eq:sys2}, we obtain
$$
  \tau \ddt x_0 + \tau^2 \ddt x_1 = 
    \tau^2 f'(y_0) y_1 
    - \tau x_1 + \mathcal{O}(\tau^2),
$$
where we used the Taylor expansion of $f(y_0 + \tau y_1 + \mathcal{O}(\tau^2))$
and the definitions \ttg{eq:defx0}.
Comparing the terms of order $\tau$, we find
that the equality is satisfied up to terms quadratic in $\tau$ if $x_1 = - \mathrm{d} x_0 / \mathrm{d}t$, which coincides with the definition 
\ttg{eq:defhatx} of $\hat x$.
Clearly, the equation for $y$ is satisfied in a similar manner.
\end{proof}

Second, we prove that the D-QSSA approximation is equivalent to
the first-order correction \ttg{eq:defhatx} up to terms cubic in $\tau$.
To this end we introduce the D-QSSA $\tilde x, \tilde y$ and approximate system \ttg{eq:sys2}
by
\begin{align}
\label{eq:deftildex}
\begin{split}  
  \tilde x(t) &= \tau f(\tilde y(t - \tau)), \\
  \ddt \tilde y(t) &= h( \tilde x(t), \tilde y(t) ).
\end{split}  
\end{align}

\begin{theorem}
The D-QSSA approximation $\tilde x$
is equivalent to the first-order correction $\hat x$ up to terms cubic in $\tau$.
\end{theorem}
\begin{proof}
Using the Taylor expansion in the definition \ttg{eq:deftildex} of the D-QSSA,
we obtain
\begin{multline*}
  \tilde x(t) = \tau f( \tilde y(t - \tau) ) 
  = \tau f(\tilde y(t)) - \tau^2 f'(\tilde y(t)) \ddt \tilde y(t) + \mathcal{O}(\tau^3)
\\  
  = \tau f(\tilde y(t)) - \tau^2 f'(\tilde y(t)) h(\tilde x(t),\tilde y(t)) 
    + \mathcal{O}(\tau^3).
\end{multline*}
This coincides with the definition \ttg{eq:defhatx} 
of $\hat x(t)$ up to the $\mathcal{O}(\tau^3)$ terms.
\end{proof}

Consequently, the D-QSSA approximation satisfies system \ttg{eq:sys2} up to terms quadratic in $\tau$. In other words, the D-QSSA approximates the invariant manifold 
of system \ttg{eq:sys2} up to terms quadratic in $\tau$.

Thus, in case of constant and small $\tau$ the D-QSSA is equivalent to the first-order
correction \ttg{eq:defhatx} and the delay can be avoided. 
However, if the delay is not small then there is no theoretical reason for this first-order correction to be accurate, but the D-QSSA still has the potential to yield acceptable accuracy.
For example, in Section~\ref{se:cellcyc} below, we present an oscillatory system, where the time varying delay grows up to 20\,\% of the period and we still observe good accuracy of the D-QSSA approximation.
Further, the D-QSSA approximation is defined and produces good results even if the term $g(t)$ depends on $t$ or on the other slow variables, which is the case of the oscillatory system in Section~\ref{se:cellcyc}. 
Furthermore, the D-QSSA \ttg{eq:deftildex} is simpler to understand and use in comparison with the first-order correction \ttg{eq:defhatx}.

\section{D-QSSA for mass action systems}
\label{se:massaction}

In this section we show that the mass action ODEs
describing kinetics of chemical systems can be, under certain assumptions,
expressed in the form \ttg{eq:ode1}.
We explain how the D-QSSA \ttg{eq:D-QSSA} can be formally applied 
to these systems and provide an explicit formula for the approximation.
We emphasize that the theory presented in Sections~\ref{se:D-QSSA} and \ref{se:system} 
is limited to special cases of mass action systems.

We will consider general chemical systems and for their description
we will use the notation inspired by \cite{Chellaboina:2009}.
Consider $\Nchs$ chemical species $X_1,\dots, X_\Nchs$ and $q$ chemical reactions
\begin{equation}
\label{eq:chemsys}
  \sum\limits_{j=1}^\Nchs A_{ij} X_j \stackrel{k_i}{\longrightarrow}
  \sum\limits_{j=1}^\Nchs B_{ij} X_j, \quad i=1,2,\dots,q,
\end{equation}
where $k_i > 0$ is the reaction rate of the $i$-th reaction. The stoichiometric coefficients $A_{ij}$ and $B_{ij}$ are assumed to be non-negative integers.
Notice that system \ttg{eq:chemsys} can be expressed in matrix-vector form as
\begin{equation}
\label{eq:chemsysvec}
  A \b{X} \stackrel{k}{\longrightarrow} B \b{X},
\end{equation}
where $\b{X} = [X_1,\dots,X_\Nchs]^\top$ is a column vector of chemical species,
$\b{k} = [k_1,\dots,k_q]^\top$ is a column vector of reaction rates, and
$A=[A_{ij}]$ and $B=[B_{ij}]$ are $q\times \Nchs$ matrices of stoichiometric coefficients.

The time evolution of concentrations $\b{x}(t) = [x_1(t), \dots, x_\Nchs(t)]^\top$ of
respective chemical species is modelled by a mass action system of
ODEs.
To express this system in a vector form,
we denote by $K\in\R^{q\times q}$
a diagonal matrix with values $k_1,\dots,k_q$ on the diagonal and
by $M = (B-A)^\top \in \R^{\Nchs\times q}$ the stoichiometric matrix.
Further, we introduce the vector-matrix exponentiation \cite{Chellaboina:2009}.
By definition, $\b{x}^A(t)$ is a vector in  $\R^q$
with entries given by $\prod_{\ell=1}^\Nchs x_\ell^{A_{i\ell}}(t)$ for $i=1,2,\dots,q$.
Now, the mass action system can be expressed as
\begin{equation}
\label{eq:massaction}
  \ddt \b{x}(t) = M K \b{x}^A(t), \quad t\geq 0
\end{equation}
with an initial condition
\begin{equation}
 \label{eq:massactioninicond}
  \b{x}(0) = \b{x}^0.
\end{equation}

Equivalently, system \ttg{eq:massaction} can be expressed in component-wise
notation as
\begin{equation}
\label{eq:massaction2}
  \ddt x_j(t) = \sum\limits_{i=1}^q M_{ji} k_i
    \prod\limits_{\ell=1}^\Nchs x_\ell^{A_{i\ell}}(t), \quad j=1,2,\dots,\Nchs.
\end{equation}

It is well known that mass action systems preserve nonnegativity, 
see for example \cite{HorJac:1972} or the recent review \cite[Thm.~2]{Chellaboina:2009}.
Thus, the solution $\b{x}(t)$ of \ttg{eq:massaction} is guaranteed 
to have non-negative entries for all $t \geq 0$
provided $\b{x}^0$ has non-negative entries.
We will use this fact in the rest of the paper.

The D-QSSA will be applied to a portion of system variables $x_j$.
Variables approximated by the D-QSSA will be called
fast, and the other variables slow.
Identification of the fast and slow variables can be done by
the standard quasi-steady state analysis, see e.g.\ \cite{SegSle:1989}.
Without loss of generality, we will assume that
variables $x_1, x_2, \dots, x_{\Nf}$ 
are fast and will be approximated by the D-QSSA.
The dynamics of the slow variables $x_{\Nf+1}, \dots, x_\Nchs$ will be
determined by the resulting system of delay differential equations.
Naturally, we require $0 < \Nf < \Nchs$.

In order to transform equations \ttg{eq:massaction2} for $j=1,2,\dots,\Nf$
to the form \ttg{eq:ode1}, we have to consider the following assumption on
the chemical system \ttg{eq:chemsys}.
\begin{enumerate}
\item[{\bfseries A1}.]
  If $A_{ij} \neq B_{ij}$ 
  then either $A_{ij} = 0$ or $A_{ij} = 1$,
  for all $i=1,2,\dots,q$ and $j=1,2,\dots,\Nf$.
\end{enumerate}

Under this assumption equations \ttg{eq:massaction2} for the fast variables
can be expressed in the form
\begin{equation}
\label{eq:massaction3}
   \ddt x_j(t) = f_j(t) - g_j(t)x_j(t), \quad j=1,2,\dots,\Nf,
\end{equation}
where
\begin{align*}
  f_j(t) &= f_j(x_1(t),\dots,x_\Nchs(t)) = \sum\limits_{i\in\cF_j} M_{ji} k_i
    \prod\limits_{\ell=1,\ell\neq j}^\Nchs x_\ell^{A_{i\ell}}(t),
\\
  g_j(t) &= g_j(x_1(t),\dots,x_\Nchs(t)) = -\sum\limits_{i\in\cG_j} M_{ji} k_i
    \prod\limits_{\ell=1,\ell\neq j}^\Nchs x_\ell^{A_{i\ell}}(t)
\end{align*}
with
$
  \cF_j = \{ i\in\{1,2,\dots,q\} : M_{ji} \neq 0 \textrm{ and } A_{ij} = 0 \}
$
  and 
$
  \cG_j = \{ i\in\{1,2,\dots,q\} : M_{ji} \neq 0 \textrm{ and } A_{ij} = 1 \}.
$

Assumption A1 guarantees that there are no quadratic and higher-order terms
with respect to $x_j$ in \ttg{eq:massaction3}.
However, it limits the class of systems to which the D-QSSA can be applied.
For example, a simple dimerization reaction $2 X_1 \rightarrow X_2$ violates
assumption A1. On the other hand, many chemical systems satisfy this assumption.
For examples see Section~\ref{se:numex} below or \cite{Vej:MB}.


Note that in general $f_j(t) = f_j(x_1(t),\dots,x_\Nchs(t))$ is a function
of $x_1(t),\dots,x_\Nchs(t)$. However, practically it does not depend on all
$x_1(t),\dots,x_\Nchs(t)$.
In fact, $f_j$ is a function of $x_\ell(t)$ for $\ell \in \wcF_j$,
where $\wcF_j =\{ \ell\neq j : \exists i \in \cF_j : A_{i\ell} \neq 0 \}$.
Namely $f_j(t) = f( \{ x_\ell(t) : \ell \in \wcF_j\})$.
Similarly, $g_j(t) = g( \{ x_\ell(t) : \ell \in \wcG_j\})$,
where $\wcG_j =\{ \ell\neq j : \exists i \in \cG_j : A_{i\ell} \neq 0 \}$.

Now, we are ready to apply the D-QSSA \ttg{eq:D-QSSA} to equations \ttg{eq:massaction3}.
As a result we obtain approximations
\begin{equation}
\label{eq:D-QSSAma}
  \tilde x_j(t) = \frac{f_j(t-\tau_j(t))}{g_j(t-\tau_j(t))},
  \quad\textrm{where } \tau_j(t) = \frac{1}{g_j(t)},\ j=1,2,\dots,\Nf.
\end{equation}
As was recognized already in Definition~\ref{de:D-QSSA}, these assumptions
are well defined only if all $g_j(t)$ are positive. Therefore we
make the following assumption.
\begin{enumerate}
\item[{\bfseries A2}.]
  Let $g_j(t) > 0$ for all $t>0$ and $j=1,2,\dots,\Nf$.
\end{enumerate}

We use approximations \ttg{eq:D-QSSAma} to reduce system \ttg{eq:massaction2}.
In order to do this in a straightforward way, we will assume
that the functions $\tilde x_j$ are independent of all $x_k$, $j,k=1,2,\dots,\Nf$.
This means that functions $f_j$ and $g_j$ for $j=1,2,\dots,\Nf$ depend on
$x_k$ for $k=\Nf+1,\dots,\Nchs$ only.
This assumption can be rigorously formulated in terms of sets
$\wcF_j$ and $\wcG_j$ as follows.
\begin{enumerate}
\item[{\bfseries A3}.]
  Let $\wcF_j \cup \wcG_j$ be such that it does not contain
  $1,2,\dots,\Nf$ for all $j=1,2,\dots,\Nf$.
\end{enumerate}
Thus, the functions $\tilde x_j$ can be expressed in terms of $x_{\Nf+1}, x_{\Nf+2}, \dots, x_\Nchs$, i.e.,
\begin{equation}
  \label{eq:txj}
  \tilde x_j(t) = \tilde x_j( x_{\Nf+1}(t), \dots, x_\Nchs(t) )
  \quad\textrm{for }j=1,2,\dots,\Nf.
\end{equation}

These relations can be substituted into \ttg{eq:massaction2} as approximations
for $x_1,x_2,\dots,x_{\Nf}$ and we obtain the reduced system
\begin{align}
  \label{eq:redsys1}
  \tilde x_j(t) &= \frac{f_j( \tilde x_{\Nf+1}(t-\tau_j(t)), \dots, \tilde x_\Nchs(t-\tau_j(t)) )}
                        {g_j( \tilde x_{\Nf+1}(t-\tau_j(t)), \dots, \tilde x_\Nchs(t-\tau_j(t)) )},
  \quad j=1,2,\dots,\Nf,
  \\
  \label{eq:redsys2}
  \ddt \tilde x_j(t) &= \sum\limits_{i=1}^q M_{ji} k_i
    \prod\limits_{\ell=1}^{\Nf} \tilde x_\ell^{A_{i\ell}}(t)
    \cdot
    \prod\limits_{\ell=\Nf+1}^\Nchs \tilde x_\ell^{A_{i\ell}}(t),
  \quad j=\Nf+1,\Nf+2,\dots,\Nchs,
\end{align}
where
\begin{equation}
  \label{eq:redsysd}
  \tau_j(t) = \frac{1}{g_j( \tilde x_{\Nf+1}(t), \dots, \tilde x_\Nchs(t) )},
  \quad j = 1,2,\dots,\Nf.
\end{equation}
System \ttg{eq:redsys1}--\ttg{eq:redsys2} is a system of $\Nchs-\Nf$ delay differential
equations with delays \ttg{eq:redsysd} dependent on the state variables
$\tilde x_{\Nf+1}(t), \dots, \tilde x_\Nchs(t)$.
To make this system solvable, we
must have values of $x_j(t)$ for $t$ negative. The simplest assumption is the
constant extension of the initial condition \ttg{eq:massactioninicond}:
$$
  x_j(t) = x^0_j \quad\textrm{for } t \leq 0
  \textrm{ and } j=\Nf+1,\Nf+2,\dots,\Nchs,
$$
where $x^0_j$ stand for the entries of $\b{x}^0$.

The fact that delays are state dependent might complicate the subsequent analysis,
but in practical cases these delays can be approximated by constants.
For an illustration of this effect, see Sections~\ref{se:numex} 
and \ref{se:cellcyc} below.

Assumption A3 is not fundamental. If it is not satisfied, then the reduction
method can still be used in a recurrent manner. We can decrease the number of fast
variables until assumption A3 is satisfied and construct the reduced system
\ttg{eq:redsys1}--\ttg{eq:redsys2} with a smaller number of fast variables.
Then we can attempt to reduce the resulting system again. In many cases
this recurrent reduction enables us to reduce the original system substantially.

Let us note that after one step of this reduction, the system need not be
in mass action form. However, as soon as the particular equation can
be expressed in the form \ttg{eq:D-QSSA} with positive $g(t)$ then
the D-QSSA can be applied. An example of the recurrent application 
of the D-QSSA is presented below in Section~\ref{se:cellcyc}.

Finally, let us mention a sufficient condition for the validity of the assumption A2
which can be useful in certain applications.
Functions $g_j(t) = g_j(x_1(t),\dots,x_\Nchs(t))$ are in general polynomials
in $x_i(t)$, $i=1,2,\dots,\Nchs$. Since $x_i(t) \geq 0$, assumption A2 corresponds 
to positivity of a multivariate polynomial in the nonnegative orthant.
Thus, a straightforward sufficient condition is nonnegativity of all coefficients
of this polynomial and positivity of at least one of its terms.

\section{Expression of Hes1 protein}
\label{se:numex}

This section shows an example of the application of the D-QSSA to a chemical system.
First, we provide details about the mass-action model and its reduction by the D-QSSA. Second, we solve the systems numerically and show the accuracy of this reduction.

\subsection{Model reduction}

Let us consider the following chemical system which was inspired by the model \cite{Monk:2003}
for expression of the Hes1 protein:
$$
  D \stackrel{\alpham}{\longrightarrow} D + M \qquad
  M \stackrel{\alphap}{\longrightarrow} M + P \qquad
  D + n P \dblarrow{\gamma_1}{\gamma_{-1}} D' \qquad
  M \stackrel{\mum}{\longrightarrow} \emptyset \qquad
  P \stackrel{\mup}{\longrightarrow} \emptyset,
$$
where $D$ corresponds to the \textsl{hes1} gene, $M$ to \textsl{hes1} mRNA,
and $P$ to Hes1 protein.
In this model, the Hes1 protein can bind to $n$ promoter sites of the gene producing an inactive complex $D'$. The rate constant $\alpham$ corresponds to transcription of the gene to the mRNA molecule, $\alphap$ to translation of the mRNA to the protein, $\gamma_1$ and $\gamma_{-1}$ are binding and unbinding rates of Hes1 to the promoter region, and $\mum$ and $\mup$ are the degradation rates of $M$ and $P$, respectively.

The law of mass action yields a system of ODEs for concentrations of $D'=D'(t)$, $D=D(t)$, $M=M(t)$, and $P=P(t)$ of the form \ttg{eq:massaction}.
We complete this system with a natural initial condition
\begin{equation}
  \label{eq:exinicond}
  D(0)=1,\quad D'(0)=0,\quad M(0)=0,\quad P(0)=0.
\end{equation}
Biologically, binding and unbinding of a transcription factor (Hes1 protein in this case)
to the promoter region of a gene is often a frequently occurring reaction
in comparison with the relatively slow processes of transcription and translation.
Therefore, in analogy with \cite{Vilar:2002}, we
consider $D$ and $D'$ to be the fast species.
With this choice, we can readily verify the validity of
assumptions A1 and A2. However, assumption A3 is not satisfied.
Fortunately, this assumption is technical only. Moreover, in this case it
can be easily overcome by elimination of one of the variables $D$ or $D'$.

Initial condition \ttg{eq:exinicond}, together with the fact that $D(t) + D'(t)$
is constant, implies $D'(t)=1-D(t)$. Thus, eliminating variable $D'(t)$
from the system yields the following three equations:
\begin{align}
\label{eq:exODE_D}
\ddt D &= \gamma_{-1} - (\gamma_{-1} + \gamma_1 P^n) D,  \\
\label{eq:exODE_M}
\ddt M &= \alpham D - \mum M,  \\
\label{eq:exODE_P}
\ddt P &= \alphap M - \mup P + n \left[\gamma_{-1} - (\gamma_{-1} + \gamma_1 P^n)D \right].
\end{align}
Notice that equation \ttg{eq:exODE_D} now satisfies all assumptions A1--A3.

It is convenient to rescale the unknowns in this system to be of comparable size.
We follow the scaling from \cite{Monk:2003}, define $m = M/\alpham$, $p=P/(\alpham\alphap)$, and transform system \ttg{eq:exODE_D}--\ttg{eq:exODE_P} to
\begin{align}
\label{eq:exODE_d}
\ddt D &= \gamma_{-1} - (\gamma_{-1} + \gamma p^n) D,  \\
\label{eq:exODE_m}
\ddt m &= D - \mum m,  \\
\label{eq:exODE_p}
\ddt p &= m - \mup p + \frac{n}{\alpha} \left[\gamma_{-1} - (\gamma_{-1} + \gamma p^n)D \right].
\end{align}
Here, $\alpha = \alphap \alpham$ and $\gamma = \gamma_1 (\alphap \alpham)^n$.

To use the standard QSSA,
we approximate $D(t)$ by its quasi-steady state approximation
\begin{equation}
  \label{eq:exQSSA_D}
  \widehat D(t) = \frac{\gamma_{-1}}{\gamma_{-1} + \gamma \whp^n(t)}.
\end{equation}
Functions $\whp(t)$ and $\whm(t)$ approximate $p(t)$ and $m(t)$,
respectively, and are determined by reducing system
\ttg{eq:exODE_d}--\ttg{eq:exODE_p} to the following two equations:
\begin{align}
\label{eq:exQSSA_M}
\ddt \whm &= \frac{\gamma_{-1}}{\gamma_{-1} + \gamma \whp^n} - \mum \whm,  \\
\label{eq:exQSSA_P}
\ddt \whp &= \whm - \mup \whp.
\end{align}

Alternatively, we can reduce system \ttg{eq:exODE_d}--\ttg{eq:exODE_p} using
the D-QSSA. Clearly, equation \ttg{eq:exODE_d} is in the
form \ttg{eq:D-QSSA} with $x(t) = D(t)$, $f(t) = \gamma_{-1}$ and
$g(t) = g( p(t) ) = \gamma_{-1} + \gamma p^n(t)$.
Thus, Definition~\ref{de:D-QSSA} yields the approximation
\begin{equation}
  \label{eq:exD-QSSA_D}
  \widetilde D(t) = \frac{\gamma_{-1}}{\gamma_{-1} + \gamma \wtp^n(t-\tau(t))},
  \quad\textrm{where }
  \tau(t) = \frac{1}{\gamma_{-1} + \gamma \wtp^n(t)}.
\end{equation}
According to \ttg{eq:redsys1}--\ttg{eq:redsys2},
system \ttg{eq:exODE_d}--\ttg{eq:exODE_p}
reduces to the following system of differential equations with delay:
\begin{align*}
\ddt \wtm(t) &= \widetilde D(t) - \mum \wtm(t),  \\
\ddt \wtp(t) &= \wtm(t) - \mup \wtp(t)
  + \frac{n}{\alpha} \left[\gamma_{-1} - (\gamma_{-1} + \gamma \wtp^n(t)) \widetilde D(t) \right].
\end{align*}
Alternatively, we can use the substitution \ttg{eq:exD-QSSA_D} and rewrite this system as
\begin{align}
\label{eq:exD-QSSA_M}
\ddt \wtm(t) &= \frac{\gamma_{-1}}{\gamma_{-1} + \gamma \wtp^n(t-\tau(t))} - \mum \wtm(t),  \\
\label{eq:exD-QSSA_P}
\ddt \wtp(t) &= \wtm(t) - \mup \wtp(t)
 + \gamma_{-1} \frac{n}{\alpha} \left[ 1 - \frac{\gamma_{-1} + \gamma \wtp^n(t)}{\gamma_{-1} + \gamma {\wtp}^n(t-\tau(t))} \right].
\end{align}
The initial condition \ttg{eq:exinicond} extended
to negative values of $t$ reads
\begin{equation}
  \label{eq:exD-inicond}
  \wtm(t)=0,\quad \wtp(t)=0
  \quad\textrm{for all } t \leq 0.
\end{equation}

At this point, it is interesting to compare the derived systems \ttg{eq:exQSSA_M}--\ttg{eq:exQSSA_P}
and \ttg{eq:exD-QSSA_M}--\ttg{eq:exD-QSSA_P} with the rescaled version of the original
delay differential equation system use in \cite{Monk:2003}:
\begin{align}
\label{eq:Monk_M}
\ddt \overline m(t) &= \frac{1}{1 + \bigl( \overline p(t-\tautr) / p_0 \bigr)^n} - \mum \overline m(t),  \\
\label{eq:Monk_P}
\ddt \overline p(t) &= \overline m(t) - \mup \overline p(t).
\end{align}
Here, $\overline m(t)$ and $\overline p(t)$ model the rescaled concentrations of
$M$ and $P$, respectively, in the same manner as solutions of systems \ttg{eq:exQSSA_M}--\ttg{eq:exQSSA_P} and \ttg{eq:exD-QSSA_M}--\ttg{eq:exD-QSSA_P}.
A distinctive feature of system \ttg{eq:Monk_M}--\ttg{eq:Monk_P}
is the delay $\tautr$, which is interpreted as the transcriptional delay.
Transcription is a complicated process which moves sequentially along the chain
of the mRNA molecule and synthesizes a protein. The time needed to complete the
synthesis can be significant and, therefore, system
\ttg{eq:Monk_M}--\ttg{eq:Monk_P} compensates for it by introducing
the delay $\tautr$.

Systems \ttg{eq:exQSSA_M}--\ttg{eq:exQSSA_P} and \ttg{eq:exD-QSSA_M}--\ttg{eq:exD-QSSA_P}
have been derived from mass action kinetics using the QSSA and the D-QSSA, respectively.
However, system \ttg{eq:Monk_M}--\ttg{eq:Monk_P} was introduced in a phenomenological  manner using the heuristic Hill function $[1 + (p/p_0)^n]^{-1}$ combined with the time delay $\tautr$ corresponding to transcription, see \cite{Monk:2003} and \cite{Hirata:2002}. Interestingly, if $p_0 = (\gamma_{-1}/\gamma)^{1/n}$ and $\tau(t) = \tautr$ then equations \ttg{eq:exD-QSSA_M} and \ttg{eq:Monk_M} are identical. Similarly, equation \ttg{eq:exQSSA_P} is the same as \ttg{eq:Monk_P}. Thus, the system \ttg{eq:Monk_M}--\ttg{eq:Monk_P} can be rigorously derived from the mass-action kinetics using the QSSA and the D-QSSA.
However, the biological meaning of the delay $\tau(t)$ in \ttg{eq:exD-QSSA_M} and the delay $\tautr$ in \ttg{eq:Monk_M} differs. The delay $\tau(t)$ compensates for the time needed to bind (or unbind) $n$ molecules of $P$ to the promoter region of \textsl{hes1} gene, while $\tautr$ corresponds to the transcriptional delay. Therefore, it is not biologically plausible to identify these two delays. From a biological viewpoint, these delays should be summed up. However, the D-QSSA methodology can be used to derive the transcriptional delay rigorously, but this would require a more detailed mass action model of transcription, for example the model from \cite{RouZhu:2006b,RouZhu:2006}.


\subsection{Numerical results}

We will numerically solve and compare system
\ttg{eq:exODE_d}--\ttg{eq:exODE_p}
with its QSSA approximation \ttg{eq:exQSSA_M}--\ttg{eq:exQSSA_P} and
its D-QSSA approximation \ttg{eq:exD-QSSA_M}--\ttg{eq:exD-QSSA_P}.
We consider parameter values from \cite{Monk:2003}. However this reference provides
values for $n$, $\mum$, $\mup$ and not for $\gamma$, $\gamma_{-1}$,
and $\alpha$, because these are irrelevant for system \ttg{eq:Monk_M}--\ttg{eq:Monk_P}. Instead, reference \cite{Monk:2003} provides
the value $p_0 = 100$. Since $p_0^n = \gamma_{-1}/\gamma$, we choose
values of $\gamma$ and $\gamma_{-1}$ to agree with this relation.
In particular, we consider
\begin{equation}
\label{eq:param1}
  n=5,\ 
  \mum = \mup = 0.03, \ 
  \gamma = 2\cdot10^{-12},\  
  \gamma_{-1} = 0.02,\ 
  \alpha = 500.
\end{equation}

The systems of ODEs \ttg{eq:exODE_d}--\ttg{eq:exODE_p}
and \ttg{eq:exQSSA_D}--\ttg{eq:exQSSA_P} can be solved by practically any
ODE solver.
We use the Matlab \texttt{ode} solver.
Numerical solution of a system with delays is straightforward.
For simplicity, we use the explicit Euler method with a sufficiently small time step.
In every time step, we compute the
delay $\tau(t)$ given by \ttg{eq:exD-QSSA_D} and use the corresponding historical value $\wtp(t-\tau(t))$ to evaluate the right-hand side of
\ttg{eq:exD-QSSA_M}--\ttg{eq:exD-QSSA_P}. Therefore, values of $\wtp$ have to be stored in every time step.

\begin{figure}
\includegraphics[width=0.48\textwidth]{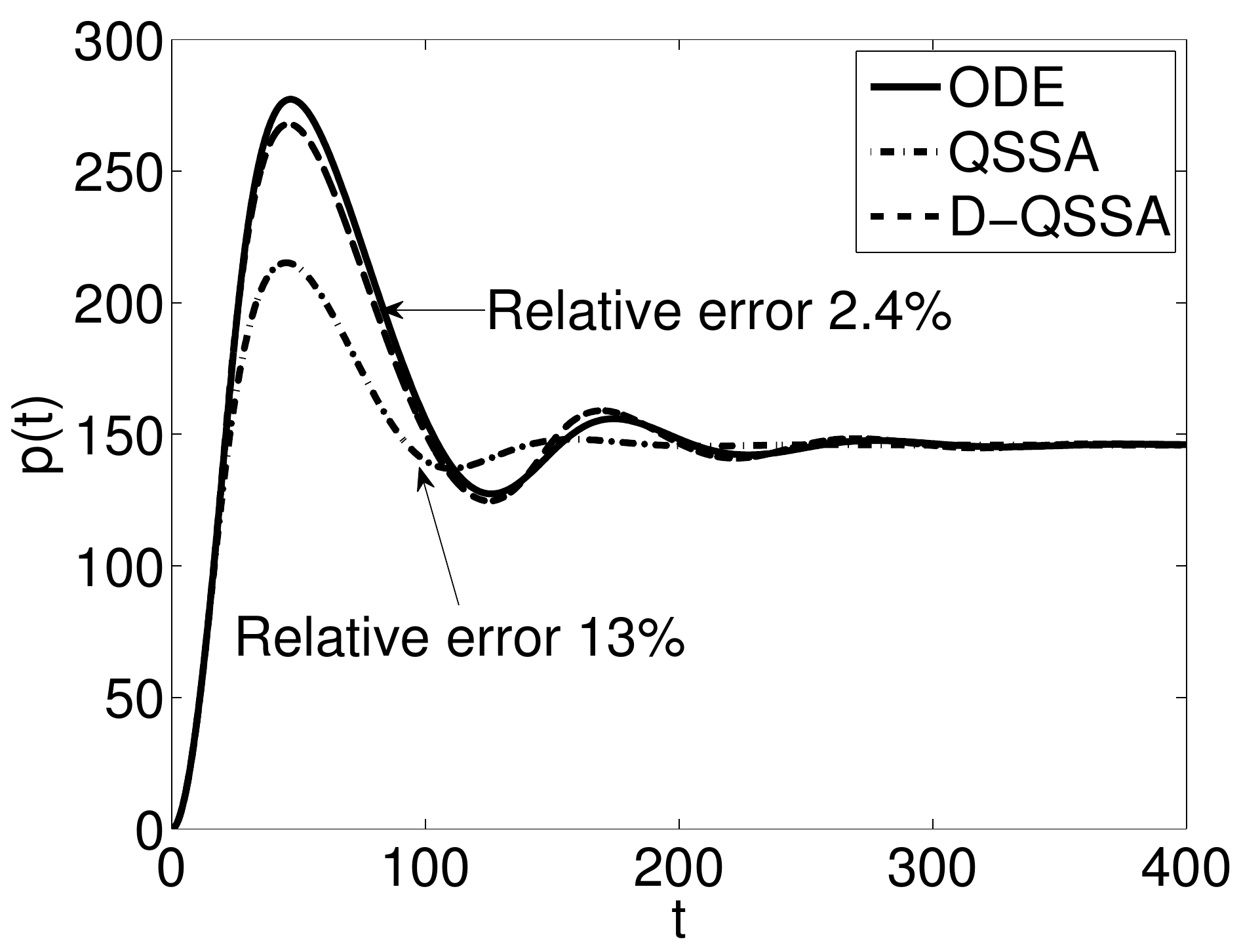}\quad
\includegraphics[width=0.48\textwidth]{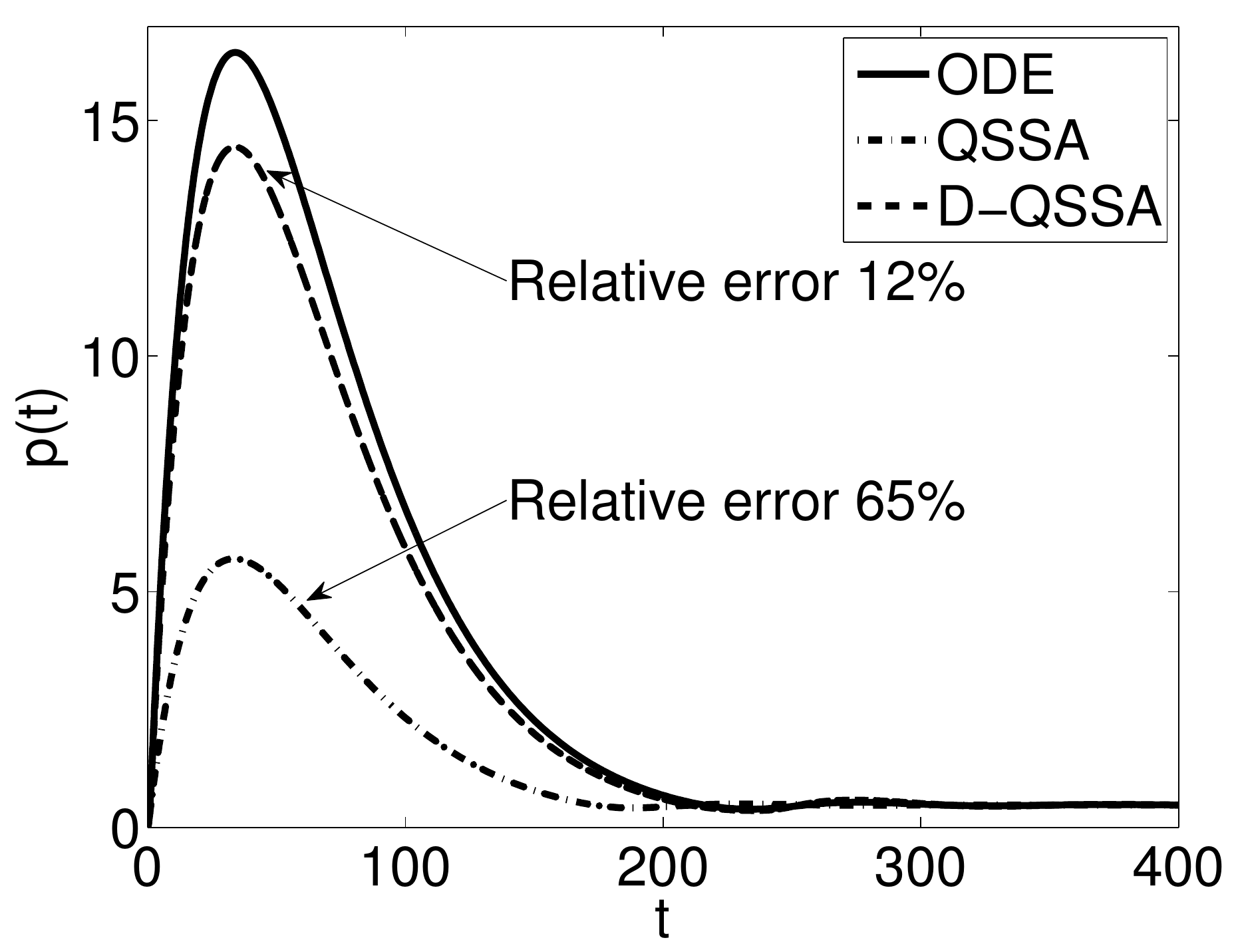}
\caption{\label{fi:exHes1}
Numerical solutions $p(t)$, $\whp(t)$, and $\wtp(t)$ of
ODE system \ttg{eq:exODE_d}--\ttg{eq:exODE_p},
QSSA system \ttg{eq:exQSSA_D}--\ttg{eq:exQSSA_P}, and
D-QSSA system \ttg{eq:exD-QSSA_D}--\ttg{eq:exD-QSSA_P}.
Left panel corresponds to parameter values \ttg{eq:param1}, the right panel to \ttg{eq:param2}.
}
\end{figure}

Figure~\ref{fi:exHes1} (left) presents numerically computed values of $p(t)$,
$\whp(t)$, and $\wtp(t)$ for $t\in[0,T]$ as given by
\ttg{eq:exODE_d}--\ttg{eq:exODE_p},
\ttg{eq:exQSSA_D}--\ttg{eq:exQSSA_P}, and
\ttg{eq:exD-QSSA_D}--\ttg{eq:exD-QSSA_P}, respectively.
We observe that the approximation provided by the D-QSSA is much more accurate
than the standard quasi-steady state approximation. Quantitatively, the $L^2$ relative error of the quasi-steady state solution $\whp$ is approximately 13\,\%,
while the relative error of the D-QSSA solution is approximately 2.4\,\%.
To be rigorous,
$$
  \frac{\| p - \whp\|_{L^2(0,T)}}{\| p \|_{L^2(0,T)}} \doteq 0.13,
  \quad\textrm{and}\quad
  \frac{\| p - \wtp\|_{L^2(0,T)}}{\| p \|_{L^2(0,T)}} \doteq 0.024.
$$

To show the quality of the QSSA and D-QSSA approximations also in different
parameter regimes, we change the values of parameters $\gamma$ and $\gamma_{-1}$ to
\begin{equation}
\label{eq:param2}
  \gamma = 10 
  \quad\textrm{and}\quad
  \gamma_{-1} = 10^{-4} 
\end{equation}
and keep the other parameters the same as in \ttg{eq:param1}.
This new choice of $\gamma$ and $\gamma_{-1}$ corresponds to
$p_0 = (\gamma_{-1}/\gamma)^{1/n} = 10$.
Figure~\ref{fi:exHes1} (right) shows the numerical solutions for parameter values
\ttg{eq:param2}. In this case the relative error of the quasi-steady state solution
is approximately 65\,\% and the error of the D-QSSA is approximately 12\,\%.

For completeness, Table~\ref{ta:relerr} presents relative errors of
all three variables in the system for both sets of parameters \ttg{eq:param1} and \ttg{eq:param2}.
This table confirms that the D-QSSA provides a considerably more accurate approximation
than the QSSA for all variables in the system.
\begin{table}\centerline{%
\begin{tabular}{l|ccc|ccc}
   & \multicolumn{3}{c|}{Parameters \ttg{eq:param1}}
   & \multicolumn{3}{c}{Parameters \ttg{eq:param2}}
\\
   & $D$ & $m$ & $p$ & $D$ & $m$ & $p$
\\ \hline
%
relative error QSSA   & 32\,\%  & 18\,\%  &  13\,\% & 77\,\%  & 65\,\%  &  65\,\% \\
relative error D-QSSA & 12\,\%  & 3.6\,\%  &  2.4\,\% & 28\,\%  & 12\,\%  &  12\,\%
\end{tabular}%
}
\caption{\label{ta:relerr}
Relative errors of QSSA and D-QSSA for all variables of the system and for both sets of parameters.}
\end{table}

The \emph{a priori} error estimate \ttg{eq:errest} provides a guaranteed bound on the error. It is accurate if the function $g(t)$ is close to constant and bounded well away from zero. However, in general it can overestimate the true error considerably, especially if the function $g(t)$ varies significantly or if it's value is close to zero.
The current example has both of these unfavourable properties. The quantities in \ttg{eq:errest} are $\varepsilon = \gamma_{-1}$, $M = \gamma_{-1} + \gamma \max_{[0,T]} |p|^n$, $f'=f'' = 0$, and $x^0 = D(0) = 1$.
Using, for illustration, parameter values \ttg{eq:param1} and estimate $\tilde p(t) \leq 300$,  
the error bound \ttg{eq:errest} yields
$$
  | D(t) - \widetilde D(t) | \leq 1.99 + 2\exp(-0.02 t),
$$
while the true error is at most 0.32 and it does not exceed 0.05 for $t>50$. Nevertheless, this shows that the D-QSSA can provide accurate results even if the error estimate \ttg{eq:errest} does not predict so.


We have also tested the influence of the last term in equation \ttg{eq:exD-QSSA_P}.\footnote{We would like to thank the anonymous referee for suggesting this test.} The reason is that removing this term corresponds to the situation when we first reduce the system by the standard QSSA and then add the delay. 
We have found that the influence of this term on the solution is very minor. Solutions with and without this term almost coincide. This is expected if the delay is small, because then the quantity evaluated at the current time and at the delayed time differ only slightly and the term resulting from the time delay is close to zero and be neglected. This means that the application of the D-QSSA to general systems of the form described in Section~\ref{se:massaction} can be in these cases simplified. The system can be reduced by the standard QSSA first and then the delays can be added.

Parameter values \ttg{eq:param1} and \ttg{eq:param2} have been chosen to show
the potential of the D-QSSA technique. In general, if we increase $\gamma_{-1}$
and correspondingly decrease $\gamma$ to keep the ratio $\gamma_{-1}/\gamma = p_0^n$
constant, then the delay $\tau(t)$ given in \ttg{eq:exD-QSSA_D} attains smaller values,
because $\tau(t) \leq 1/\gamma_{-1}$.
Eventually, it is so small that the D-QSSA and QSSA approximations practically
coincide. This is the situation when the two time scales in the system are
well separated and both approaches provide very accurate approximation.
However, in numerical tests we performed in this regime,
the D-QSSA was always more accurate than the QSSA.
On the other hand, if we decrease $\gamma_{-1}$ and increase $\gamma$
to keep the ratio $\gamma_{-1}/\gamma = p_0^n$ constant, then
the delay $\tau(t)$ attains high values and the D-QSSA
becomes very inaccurate. However, this is the situation when
the time scales are not well separated and the standard QSSA is not accurate
either.
This behaviour of the D-QSSA is well explained by the dependence of
the error estimate \ttg{eq:errest} on $\varepsilon$.

Finally, we note that all these results were achieved with the delay as defined in \ttg{eq:exD-QSSA_D}.
This delay is state dependent, because it depends on the value $\wtp(t)$.
This might be a hindrance and therefore we have tried to replace the state
dependent delay by a constant delay.
Based on experimental fitting to the system \ttg{eq:exODE_D}--\ttg{eq:exODE_P},
we have found that even a constant delay can provide very accurate results.
Choosing $\tau(t)=6$ for parameter set \ttg{eq:param1} and $\tau(t)=0.89$
for parameters \ttg{eq:param2}, we achieve relative errors in $p$
of approximately 1\,\%. However, the accuracy of these approximations is quite
sensitive to the value of the delay and so far it is not clear how to determine
the optimal constant value of the delay.

\section{Cell cycle model}
\label{se:cellcyc}

The D-QSSA approach seems to be especially suitable for oscillating systems.
In this section, we provide an example of the application of the D-QSSA to a cell-cycle model and we will see that the standard QSSA destroys the oscillations, but the reduction by the D-QSSA keeps them. In addition, we use this example to show the application of the D-QSSA in a recurrent manner as was mentioned in Section~\ref{se:massaction}. 
Further, we will use this example to illustrate the potential of simplifying the time varying delays to a constant delays and compare various possible choices of the constant delay.

\subsection{System reduction}

We consider the cell cycle model from \cite{FerTsaYan:2011}:
\begin{align}
\nonumber
\ddt C(t) &= \alpha_1 - \beta_1 C(t) \frac{A^{n_1}(t)}{K_1^{n_1}+A^{n_1}(t)},
\\  \label{eq:ccorig}
\ddt P(t) &= \alpha_2 (1-P(t)) \frac{C^{n_2}(t)}{K_2^{n_2}+C^{n_2}(t)} - \beta_2 P(t),
\\  \nonumber
\ddt A(t) &= \alpha_3 (1-A(t)) \frac{P^{n_3}(t)}{K_3^{n_3}+P^{n_3}(t)} - \beta_3 A(t)
\end{align}
with parameter values
%
$
\alpha_1 = 0.1, \  \alpha_2 = \alpha_3 = 3, \
\beta_1 = 3, \ \beta_2 = \beta_3 = 1, \
K_1 = K_2 = K_3 = 0.5, \
n_1 = n_2 = n_3 = 8.
$
In this model, $C$, $P$, and $A$ correspond to
the cyclin-dependent protein kinase (CDK1),
the polo-like kinase 1 (Plk1), and
the anaphase-promoting complex (APC), respectively.
We solve this system numerically, using the zero initial condition. 
Figure~\ref{fi:cellcyc} shows the solution component $C(t)$ as the solid line.
Clearly, the system exhibits sustained oscillations.

All three variables in system \ttg{eq:ccorig} vary on a comparable
timescale and, thus, a straightforward application of the QSSA to one of
these variables can hardly be successful. However, just for further
comparison, we first reduce the system by the standard QSSA for $P$:
\begin{align}
\nonumber
\ddt C(t) &= \alpha_1 - \beta_1 C(t) \frac{A^{n_1}(t)}{K_1^{n_1}+A^{n_1}(t)},
\\  \label{eq:ccQSSA}
\ddt A(t) &= \alpha_3 (1-A(t)) \frac{\Ps^{n_3}(t)}{K_3^{n_3}+\Ps^{n_3}(t)} - \beta_3 A(t),
\\ \nonumber
\Ps(t) 
  &= \left[ 1 + \frac{\beta_2}{\alpha_2} \left(1 + \frac{K_2^{n_2}}{C^{n_2}(t)} \right) \right]^{-1}.
\end{align}
The solution component $C(t)$ of this reduced system is presented in Figure~\ref{fi:cellcyc} by the dotted line. The QSSA for $P$ destroyed the oscillations.

Similarly, we now apply the D-QSSA to $P$. The ODE for $P$ in \ttg{eq:ccorig},
can be easily expressed in the form \ttg{eq:ode1} and the D-QSSA can be formally applied. We obtain a system of two differential equations with a delay:
\begin{align}
\nonumber
\ddt C(t) &= \alpha_1 - \beta_1 C(t) \frac{A^{n_1}(t)}{K_1^{n_1}+A^{n_1}(t)},
\\  \label{eq:ccDQSSA}
\ddt A(t) &= \alpha_3 (1-A(t)) \frac{\Psd^{n_3}(t)}{K_3^{n_3}+\Psd^{n_3}(t)} - \beta_3 A(t),
\\ \nonumber
\Psd(t) &= \left[ 1 + \frac{\beta_2}{\alpha_2} \left(1 + \frac{K_2^{n_2}}{C^{n_2}(t-\tau)} \right) \right]^{-1},
\quad \tau = \left[\alpha_2 \frac{C^{n_2}(t)}{K_2^{n_2}+C^{n_2}(t)} + \beta_2 \right]^{-1}.
\end{align}
As the initial history we extend the zero values for $t=0$ to all $t<0$
and solve this delay system numerically. 
The dashed line in Figure~\ref{fi:cellcyc} shows the solution component $C(t)$. In this case, we do observe sustained oscillations, although their period and amplitude do not accurately correspond to the original system \ttg{eq:ccorig}.

Finally, we will illustrate the capabilities of the D-QSSA approach and reduce the system even more. This is an example of the recurrent application of the D-QSSA, as was mentioned in Section~\ref{se:massaction}.
We simply consider system \ttg{eq:ccDQSSA}, where the variable $P$ is approximated by the D-QSSA, and apply the D-QSSA to $A$. The delay differential equation for $A$ in \ttg{eq:ccDQSSA} can be formally written in the form \ttg{eq:ode1} and the D-QSSA can be formally used. 
The result is a single differential equation with two delays:
\begin{align}
\nonumber
\ddt C(t) &= \alpha_1 - \beta_1 C(t) \frac{\Asd^{n_1}(t)}{K_1^{n_1}+\Asd^{n_1}(t)},
\\ \label{eq:ccDQSSA2}
\Asd(t) &= \left[ 1 + \frac{\beta_3}{\alpha_3} \left(1 + \frac{K_3^{n_3}}{\Psd^{n_3}(t-\tau_2)} \right) \right]^{-1}, \quad
\tau_2 = \left[\alpha_3 \frac{\Psd^{n_3}(t)}{K_3^{n_3}+\Psd^{n_3}(t)} + \beta_3 \right]^{-1},
\\ \nonumber
\Psd(t) &= \left[ 1 + \frac{\beta_2}{\alpha_2} \left(1 + \frac{K_2^{n_2}}{C^{n_2}(t-\tau_1)} \right) \right]^{-1}, \quad
\tau_1 = \left[\alpha_2 \frac{C^{n_2}(t)}{K_2^{n_2}+C^{n_2}(t)} + \beta_2 \right]^{-1}.
\end{align}
We initialize the history in the same way as above and solve this delay equation numerically. The dash-dot line in Figure~\ref{fi:cellcyc} shows the solution.
Clearly, even the recurrent application of the D-QSSA yields a qualitatively correct approximation, which still oscillates.


\begin{figure}
\includegraphics[width=0.8\textwidth]{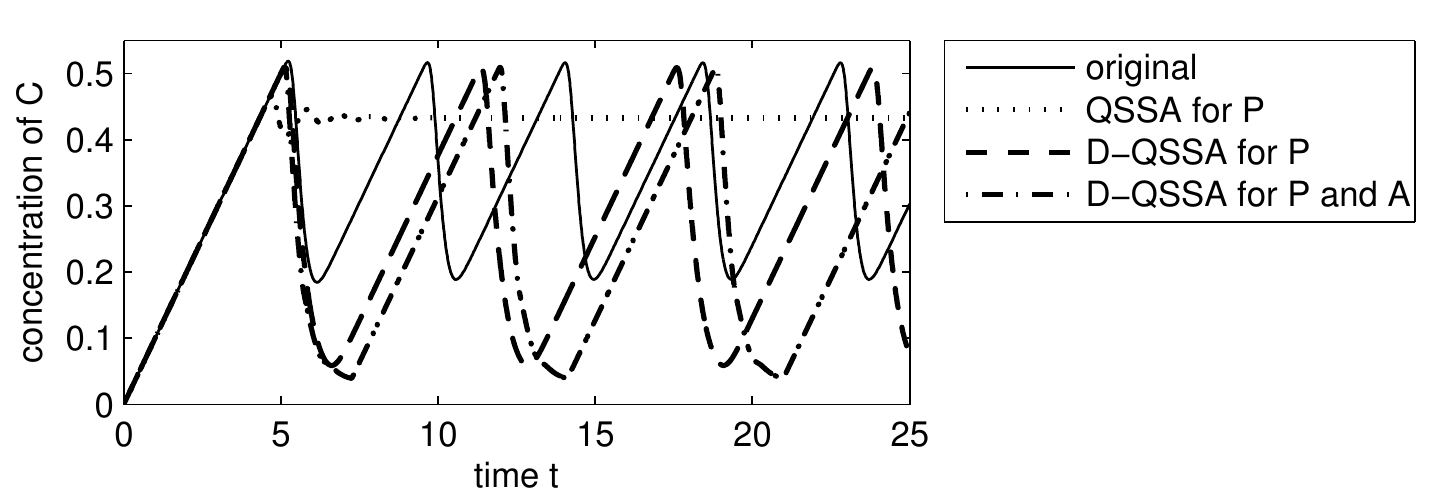}
\caption{\label{fi:cellcyc}
Solution of the original system \ttg{eq:ccorig} (solid line), system reduced by the standard QSSA for $P$ \ttg{eq:ccQSSA} (dotted line),
system reduced by the D-QSSA for $P$ \ttg{eq:ccDQSSA} (dashed line) and system reduced by the D-QSSA for both $P$ and $A$ \ttg{eq:ccDQSSA2} (dash-dot line).
}
\end{figure}

\subsection{Effects of constant delays}

As we have already mentioned, the time varying delays can be replaced by constant delays. If the constant delay is suitably chosen, the quality 
of the approximation can even increase. However, a suitable value
for the delay is difficult to find. Here, we derive several possibilities 
for choosing the constant delay based on knowledge of 
the solution of the original system \ttg{eq:ccorig}.
Moreover, the quality of the resulting approximations has to be assessed
using again the solution of the original problem. It is not clear how
to find a suitable value of the constant delay a priori, i.e. without
knowledge of the original solution.

All equations in system \ttg{eq:ccorig} can be formally written in the form
\ttg{eq:ode1}. For a given equation, the delay is formally given as $\tau(t) = 1/g(t)$, where $g(t)$ depends on the other components of the system. 
Thus, if we have the solution of this system, we know the time evolution of the time
varying delay. From this evolution, we can derive constant delays in several ways. Natural choices are the minima, maxima and mean values of $\tau(t)$
over the range of $t$.

Figure~\ref{fi:cellcyc_tau} shows the results for the system of two differential equations \ttg{eq:ccDQSSA} with a delay. We can compare the solution of the original system with the reduced systems, where the constant delay is chosen as the minimal, mean and maximal value of $\tau(t)$. The values are approximately 
$0.37$, $0.73$, and $1.00$, respectively.
We observe that the differences among these three possibilities are relatively high
and the minimal delay yields the most accurate results out of these three options.
However, choosing an ad hoc value 0.3 for the delay yields even better agreement with
the solution of the original problem. 

Figure~\ref{fi:cellcyc_tau2} shows similar results for the differential equation
\ttg{eq:ccDQSSA2} with two delays. We choose both delays as the minimal, mean, and maximal values. The delay $\tau_1$ applies to the variable $P$ and as such it has the same values as above, i.e. $0.37$, $0.73$, and $1.00$, respectively. 
The values of $\tau_2$ are approximately $0.32$, $0.72$, and $1.00$.
Again, the minima provide the best results out of these three options. 
However, the ad hoc values $\tau_1 = 0.3$ and $\tau_2 = 0.28$ yield 
an even better approximation.

\begin{figure}
\includegraphics[width=0.8\textwidth]{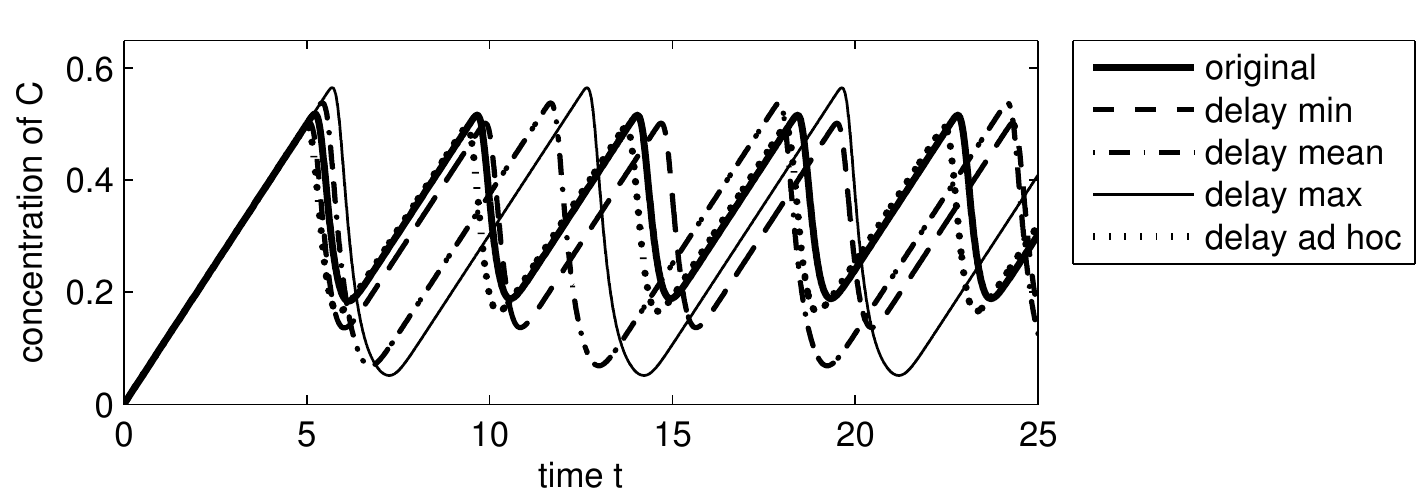}
\caption{\label{fi:cellcyc_tau}
The thick solid line indicates the solution of the original system \ttg{eq:ccorig}.
The remaining lines correspond to the system \ttg{eq:ccDQSSA} of two differential equations with various constant values of the delay:
minimal (dashed), mean (dash-dot), maximal (thin solid), and ad hoc value (dotted).
}
\end{figure}

\begin{figure}
\includegraphics[width=0.8\textwidth]{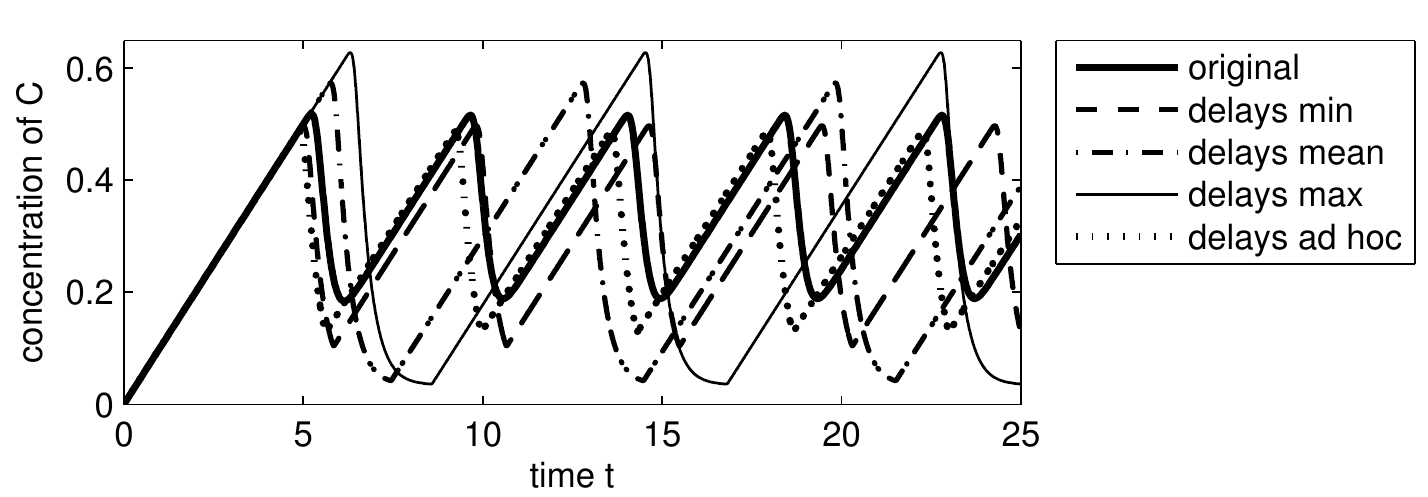}
\caption{\label{fi:cellcyc_tau2}
The thick solid line indicates the solution of the original system \ttg{eq:ccorig}.
The remaining lines correspond to the differential equation \ttg{eq:ccDQSSA2} with various constant values of the two delays. Both these delays are chosen as
the minimal (dashed), the mean (dash-dot), the maximal (thin solid), and the ad hoc value (dotted).
}
\end{figure}

In order to quantify the accuracy of various choices of the delays we present
Table~\ref{ta:ccerr}, where we provide relative errors in the period and amplitude of
the quantity $C$ with respect to the solution of the original system \ttg{eq:ccorig}.
The relative errors are expressed in percent and we observe that the choice of
the delay has a significant influence on the accuracy. 
In this case, the automatic choice of the time-dependent delay yields roughly
as accurate results as the mean delay computed from the solution of the original system.
However, more accurate results are obtained by choosing the minimal delay, but even this choice is far from optimal, as we can see from the ad hoc choice of the delay.

\begin{table}
\begin{tabular}{c|ccccc}
Error in period/amplitude (\%) & $\tau(t)$ & $\min \tau(t)$ & $\operatorname{mean} \tau(t)$ & $\max \tau(t)$ & ad hoc 
\\ \hline
two delay equations \ttg{eq:ccDQSSA} & 42/37 & 10/11 & 42/43 & 59/57 & 0.5/0.3
\\
one delay equation \ttg{eq:ccDQSSA2} & 57/43 & 10/20 & 60/62 & 88/81 & 0.3/10
\end{tabular}
\caption{Relative errors in the period and amplitude of $C$ expressed in percent
and computed with respect to the solution of the original system \ttg{eq:ccorig}.
Columns correspond to the time-dependent delay \ttg{eq:D-QSSA}, minimal, mean, maximal, and ad hoc values of the delays, respectively.
}
\label{ta:ccerr}
\end{table}

%

\section{Discussion and conclusions}
\label{se:concl}

In this paper, we present the technique of D-QSSA,
prove a corresponding error estimate, and analyse how it approximates 
the invariant manifold.
The technique is well justified for equations of the form \ttg{eq:ode1}.
We also show explicit application of the D-QSSA to a general mass action model of a general chemical system. However, the D-QSSA can be applied only if assumptions A1--A3 are satisfied.
Assumption A1 means that the stoichiometric coefficients of the fast chemical
species on the reactant side of the chemical equation cannot be greater than one, unless the species is a catalyst.
This is not as severe a restriction as it may first look, because many
biochemical reactions are of this type and, moreover, the assumption concerns fast
species only.

Assumption A2 requires positivity of functions $g_j(t)$ in the equations for the fast species. This is a technical assumption whose validity can be guaranteed from the stoichiometry and rate constants of the chemical system.



Assumption A3 means that the fast variables are independent in the sense
that reactants of all reactions, where a fast species is either produced or
consumed, do not involve any other fast species. This tends to be satisfied
in biochemical systems if the fast species correspond to genes, because
genes do not directly influence each other. Moreover, this assumption is
not fundamental, due to the possibility of recurrent application of
D-QSSA (see Section~\ref{se:massaction} and the example in Section~\ref{se:cellcyc}).

To summarise, these assumptions are not too
restrictive and make the technique of D-QSSA applicable to the majority
of mass action models of biochemical systems.

The main idea of the D-QSSA is the introduction of a time delay
to improve accuracy.
The standard QSSA ignores the time needed by fast variables to reach their
steady states. This may result in considerable errors.
The D-QSSA compensates for this time delay and improves the accuracy of
the approximation. In the example presented in Section~\ref{se:numex},
the D-QSSA exhibits more than five times smaller error
in comparison with the standard QSSA.

In comparison with the standard QSSA, the D-QSSA seems to be especially useful
for oscillating systems. The standard QSSA usually causes considerable errors
in both the period and amplitude of oscillations \cite{Vilar:2002}.
The D-QSSA enables this error to be reduced substantially.
In \cite{Vej:MB} we have successfully applied the technique of D-QSSA
to a system for modelling circadian rhythms \cite{Vilar:2002}.
This system is described by nine
ODEs and it can be reduced to two. The
standard QSSA yields roughly 30\,\% error in the
period, while the D-QSSA approach approximates the period with errors of about
1--2\,\%.

The D-QSSA can be applied even in cases when the standard QSSA 
destroys the oscillations completely. 
In Section~\ref{se:cellcyc}, we provide an example of such a successful application
of the D-QSSA. The original cell-cycle model consists of three ODEs and it has been reduced to one delay differential equation which yields sustained oscillations.
This indicates that the oscillations in this cell-cycle model are driven by 
delayed negative feedback.

The D-QSSA and the proposed choice of the delays are in agreement with 
the conclusions of \cite{RouRou:2001} where systems with constant coefficient $g$ are considered. The delays derived in \cite{RouRou:2001}, by a careful analysis of two specific examples, are identical to the delays obtained by the D-QSSA. 


The case of varying delays may
be problematic for both numerical solution and subsequent analysis.
However, numerical experiments both in this paper and in \cite{Vej:MB}
show that a constant approximation of the varying delay 
yields accurate results as well. Nevertheless, the quality of
this approximation is quite sensitive to the value of the constant delay and its optimal
choice is not clear. To make the constant delay approach practical, subsequent
research is necessary.

The essential first step for both the standard QSSA and the D-QSSA is the
identification of the fast variables. However, in some systems none of the
variables can be considered as fast, while a suitable combination can.
Reference \cite{LeeOth:2010} shows how to identify such combinations and how to
apply the QSSA to these variables. The technique of the D-QSSA can be applied
in these cases as well and we will investigate this possibility in future
research.

In chemical systems, the delay in the D-QSSA depends also on the rate constants
of the chemical reactions involved. Thus, the technique of the D-QSSA can be used in
situations where complex chemical processes are modelled by a simple reaction
with a time delay \cite{Bernard:2006,Monk:2003,RouZhu:2006},
to determine and analyse how the delay
actually depends on various parameters of the system. This promises
new insight and understanding of models with time delays. We plan, in future work,
to apply the D-QSSA to a detailed chemical model of transcription \cite{RouZhu:2006b,RouZhu:2006} to
derive and analyse the dependence of the transcriptional delay on the rates
of the elementary chemical reactions which comprise the process.

\section*{Acknowledgements}
We would like to thank the anonymous referees for their very helpful comments 
which led to substantial improvement of the original manuscript.

The research leading to these results has received funding from the People Programme (Marie Curie Actions) of the European Union's Seventh Framework Programme (FP7/2007-2013) under REA grant agreement no. 328008
and
from the European Research Council under the {European
Community's} Seventh Framework Programme {(FP7/2007-2013)} /
ERC {grant agreement} No. 239870.
Tom\'a\v{s} Vejchodsk\'y gratefully acknowledges the support of RVO 67985840.
Radek Erban would like
to thank the Royal Society for a University Research
Fellowship; Brasenose College, University of Oxford, for a Nicholas
Kurti Junior Fellowship; and the Leverhulme Trust for a Philip
Leverhulme Prize.




\bibliography{DQSSA}

\end{document}